\documentclass[oneside,reqno,english]{amsart}
\usepackage[T1]{fontenc}
\usepackage[utf8]{inputenc}
\usepackage{xcolor}
\usepackage{babel}
\usepackage{prettyref}
\usepackage{amstext}
\usepackage{amsthm}
\usepackage{amssymb}
\usepackage[pdfusetitle,
 bookmarks=true,bookmarksnumbered=false,bookmarksopen=false,
 breaklinks=false,pdfborder={0 0 0},pdfborderstyle={},backref=false,colorlinks=false]
 {hyperref}
\hypersetup{
 colorlinks=true,citecolor=blue,linkcolor=blue,linktocpage=true}

\makeatletter
\numberwithin{equation}{section}
\numberwithin{figure}{section}

\usepackage{prettyref}

\newrefformat{cor}{Corollary~\ref{#1}}
\newrefformat{subsec}{Section~\ref{#1}}
\newrefformat{lem}{Lemma~\ref{#1}}
\newrefformat{thm}{Theorem~\ref{#1}}
\newrefformat{sec}{Section~\ref{#1}}
\newrefformat{chap}{Chapter~\ref{#1}}
\newrefformat{prop}{Proposition~\ref{#1}}
\newrefformat{exa}{Example~\ref{#1}}
\newrefformat{tab}{Table~\ref{#1}}
\newrefformat{rem}{Remark~\ref{#1}}
\newrefformat{def}{Definition~\ref{#1}}
\newrefformat{fig}{Figure~\ref{#1}}
\newrefformat{claim}{Claim~\ref{#1}}
\newrefformat{assu}{Assumption~\ref{#1}}
\newrefformat{prob}{Problem~\ref{#1}}

\makeatother

\providecommand\remarkname{Remark}
\theoremstyle{remark}
\newtheorem*{rem*}{\protect\remarkname}
\providecommand\theoremname{Theorem}
\theoremstyle{plain}
\newtheorem{thm}{\protect\theoremname}[section]
\providecommand\lemmaname{Lemma}
\newtheorem{lem}[thm]{\protect\lemmaname}
\providecommand\propositionname{Proposition}
\newtheorem{prop}[thm]{\protect\propositionname}
\theoremstyle{remark}
\newtheorem{rem}[thm]{\protect\remarkname}
\begin{document}
\subjclass[2020]{Primary 42C15; Secondary 47B15, 47B06}
\title{Normalized Orbit Frames and Essential Spectrum,~II}
\begin{abstract}
Given $0<r_{j}<1$ with $r_{j}\to1$, there exist a normal operator
$T$ on $\ell^{2}$ with eigenvalues $\lambda_{j}$ satisfying $|\lambda_{j}|=r_{j}$
and $g\in\ell^{2}$ such that 
\[
\left\{ \frac{T^{n}g}{\left\Vert T^{n}g\right\Vert }\right\} _{n\geq0}\text{is a frame}\:\Longleftrightarrow\:\sum_{j\geq1}\left(1-r^{2}_{j}\right)=\infty.
\]
Whenever this condition holds, the frame bounds can be chosen with
ratio arbitrarily close to one. 
\end{abstract}

\author{James Tian}
\address{Mathematical Reviews, 535 W. William St, Suite 210, Ann Arbor, MI
48103, USA}
\email{james.ftian@gmail.com}
\keywords{normalized orbit frames, normal operators, eigenvalue distribution,
dynamical sampling, Riesz sequences}
\thanks{The author thanks Dr. Haixia Zhou for numerous helpful discussions,
as well as the amazing staff Aaron, Alice and Clara. }
\maketitle

\section{Introduction}\label{sec:1}

A normalized orbit of a self-adjoint operator cannot be a frame, but
for a normal operator it is possible. This was pointed out only recently.
Both operators have a spectral representation, quite rigid compared
with general bounded operators. Yet allowing nonreal spectrum changes
the answer.

The problem comes from dynamical sampling. Instead of measuring a
vector against a collection of different vectors, one fixes a vector
$g$ and repeatedly applies an operator $T$. The measurements are
then taken along the orbit $\{T^{n}g\}_{n\geq0}$. The idea was introduced
in \cite{MR3027915} and developed in \cite{MR3613395}. For a recent
survey on the subject, see e.g., \cite{MR5071919}. 

For a bounded operator $T$, one can ask whether $\{T^{n}g\}_{n\geq0}$
is a frame. One can also normalize each vector and ask the same question
for 
\[
\left\{ T^{n}g/\left\Vert T^{n}g\right\Vert \right\} _{n\geq0}.
\]
In the second one, the lengths of the iterates have disappeared, only
their directions remain. It is not clear at first whether these directions
can give stable expansions of every vector in the Hilbert space.

For self-adjoint operators the answer is negative \cite{MR3579135}.
It was conjectured that the same should hold for normal operators
\cite{MR5071919}. That turned out to be false. A recent construction
\cite{arXiv2606.20848} gives a diagonal normal operator with a normalized
orbit frame. Its eigenvalues are placed in finite blocks on circles
approaching the unit circle. The blocks are chosen with considerable
care.

So normal operators can have normalized orbit frames, but which ones?
There are two issues here. One is whether the frame condition places
restrictions on the spectral representation of $T$. The other is
whether those restrictions are sufficient to construct examples. 

An earlier paper \cite{arXiv2609.23804} begins with an arbitrary
bounded normal operator $T$ on a separable Hilbert space $H$. Suppose
that 
\[
x_{n}=\frac{T^{n}g}{\left\Vert T^{n}g\right\Vert },\quad n\geq0,
\]
is a frame for $H$. Since the orbit is complete, $g$ is cyclic.
The spectral theorem then represents $T$ as multiplication by $z$
on $L^{2}(\mu)$, with $g$ corresponding to the constant function
$1$. The normalized orbit becomes 
\[
\frac{z^{n}}{\sqrt{m_{n}}},\quad m_{n}=\int|z|^{2n}d\mu\left(z\right).
\]
Thus the question about a normal operator becomes a question about
normalized monomials in $L^{2}(\mu)$.

This representation itself does not immediately simplify the problem
though. The measure $\mu$ is initially arbitrary. It may have a continuous
part, and its support may occupy a large portion of the disk. Moreover,
the normalization depends on the measure itself. Changing $\mu$ changes
both the Hilbert space and the numbers $m_{n}$. This is quite different
from asking whether a given sequence is a frame in a fixed Hardy or
Bergman space.

A result of \cite{arXiv2609.23804} is that the frame condition leaves
very little freedom in the type of measure that can occur. The measure
must be purely atomic. Its atoms lie strictly inside the outer spectral
circle, and they have no accumulation point in the open disk. After
rescaling the outer radius to one, we may therefore write 
\[
\mu=\sum_{j\geq1}w_{j}\delta_{\lambda_{j}},
\]
$w_{j}>0$, $|\lambda_{j}|<1$, where the $\lambda_{j}$ are distinct
and can accumulate only on the unit circle. The operator is unitarily
equivalent to 
\[
Te_{j}=\lambda_{j}e_{j}\;\text{on }\ell^{2}.
\]
The general normal operator problem is then reduced to a problem about
sequences of complex numbers in the disk, together with positive weights.

A natural question is what those sequences can look like. We approached
this through the essential spectrum \cite{arXiv2609.23804}. For a
diagonal operator of the above form, the essential spectrum consists
of the accumulation points of the eigenvalues on the unit circle.
We proved that a closed set $E$ on the unit circle occurs as the
essential spectrum of a normal operator admitting a normalized orbit
frame if and only if $E$ has positive Haar measure. 

On the other hand, essential spectrum tells us where eigenvalues accumulate,
not how they get there. Two diagonal operators can have the same essential
spectrum while their eigenvalues approach the unit circle at entirely
different rates. One sequence may contain many eigenvalues on a given
circle, another very few. The essential spectrum sees none of this.
So the construction there is free to choose the eigenvalues as needed.
In this paper, we consider instead how much of the construction survives
if the eigenvalue magnitudes have been specified from the outset. 

Specifically, let $0<r_{j}<1$, with $r_{j}\to1$. We want a diagonal
normal operator 
\[
Te_{j}=r_{j}e^{i\theta_{j}}e_{j}
\]
whose normalized orbit is a frame. The numbers $r_{j}$ are given.
We may choose $\theta_{j}$, and the initial vector $g$, but we cannot
change the radii or introduce additional eigenvalues.

The radii may approach one very slowly, or very rapidly along certain
subsequences. A single radius may occur many times. There may be large
gaps between successive distinct radii. The earlier constructions
arrange the eigenvalues in a convenient way. Here that freedom has
disappeared. Therefore, it would be reasonable to expect additional
conditions on the distribution of the radii, or some sort of separation.
But what seems to hold is much stronger. 

We prove that there are complex numbers $\lambda_{j}$ with $|\lambda_{j}|=r_{j}$
and a vector $g\in\ell^{2}$ such that the normal operator $Te_{j}=\lambda_{j}e_{j}$
has a normalized orbit frame if and only if 
\[
\sum_{j\geq1}\left(1-r^{2}_{j}\right)=\infty.
\]
Moreover, whenever the series diverges, we can choose $T$ and $g$
so that the frame bounds are $Q(1-\varepsilon)$ and $Q(1+\varepsilon)$,
where $Q$ depends only on $\varepsilon$, not on the prescribed radii.

The condition has a classical root. For ordinary orbits of diagonal
normal operators, the frame problem is connected with interpolation
in the Hardy space. The Blaschke condition, which appears naturally
in that theory, requires convergence of the corresponding series \cite{MR117349,MR0133446}.
Here we have divergence. It is a curious reversal, although the ordinary
orbit frame problem involves more than the Blaschke condition alone.

One difficulty is that the normalization depends on everything we
choose. If $g=\sum_{j}g_{j}e_{j}$, then 
\[
\left\Vert T^{n}g\right\Vert ^{2}=\sum_{j}|g_{j}|^{2}r^{2n}_{j}.
\]
Changing the magnitude of one coefficient of $g$ changes the normalization
at every time. The radii cannot be treated independently, even though
$T$ is diagonal.

We deal with this first, before choosing any $\theta$ for the eigenvalues.
An elementary minimization gives suitable positive coefficients. Then
Fourier orthogonality and a Riesz selection theorem allow us to choose
the angles. The selection theorem is a substantial ingredient in the
proof. The passage from finite sequences to the prescribed infinite
sequence needs a separate argument, since the coefficients and the
normalized columns must both survive the limit. The construction is
developed in Sections \ref{sec:2}--\ref{sec:4}. The main result
is stated and proved in \prettyref{thm:4-3}.

\subsection*{Related literature}

Operator orbits have been studied for a long time in linear dynamics.
The early questions concern cyclic vectors, dense orbits, and the
behavior of repeated applications of a single operator \cite{MR0241956,MR1111569}.
Frame generation asks for something stronger. An orbit must be complete,
but it must also satisfy two uniform inequalities. This places restrictions
on both the operator and the initial vector.

The question of representing frames as operator orbits has been considered
in several settings. There are results on bounded operator representations,
duality, and stability \cite{MR3694620}, on representations of structured
frames \cite{MR4208084}, and on approximate representations by iterated
systems \cite{MR4361596}. Continuous parameter versions lead to frames
generated by operator semigroups \cite{MR4943786}. There are also
connections with invariant and hyperinvariant subspaces \cite{MR4990268}.

For normal operators, related work concerns iterative actions \cite{MR3579135},
Bessel orbits \cite{MR3582259}, and sampling on finite index sets
\cite{MR4093918}. Several generating vectors can also be handled
through model spaces \cite{MR4198534}.

For diagonal normal operators, ordinary orbit frames are closely related
to interpolation in the disk. The eigenvalues become interpolation
nodes. This connection has been used in the study of Carleson frames
and their redundancy \cite{MR4735722,MR4961034}. The underlying function
theory goes back to classical interpolation in Hardy spaces \cite{MR117349,MR0133446}.
Interpolation in model spaces and Bergman spaces, and the corresponding
density conditions, provide related methods \cite{MR0827223,MR0654483,MR1223222,MR1758653}.

Normalized orbits are less directly connected with a fixed function
space. The numbers $\left\Vert T^{n}g\right\Vert ^{2}$ form a moment
sequence; see e.g., \cite{MR0184042} for the classical moment problem
and related background. But the measure that determines these moments
also determines the vectors whose frame properties we are studying.
This is one reason why the usual interpolation arguments do not apply
without modification.

There is a related line of work on rescaling iterative systems and
on the spectra of normal operators that generate such systems \cite{MR4784098,arXiv2511.15625}.

Finally, our proof uses a result from the theory surrounding the Kadison--Singer
problem. The theorem of Marcus, Spielman, and Srivastava \cite{MR3374963}
led to strong partition and selection results for Bessel sequences.
One such result \cite{MR4013969} allows a Riesz sequence to be chosen
by taking one vector from each of a collection of sufficiently large
finite sets. In our case these sets consist of Fourier modulations
of vectors determined by the prescribed radii.

\section{Choosing the coordinates}\label{sec:2}

Let $\left\{ e_{j}\right\} _{j\geq1}$ be the standard orthonormal
basis (ONB) of $\ell^{2}$. Consider $0<r_{j}<1$, with $r_{j}\to1$.
We want to choose $\theta_{j}$ and $g=\left(g_{j}\right)\in\ell^{2}$
so that the diagonal operator 
\[
Te_{j}=r_{j}e^{i\theta_{j}}e_{j}
\]
has a normalized orbit which is a frame. We first choose the sizes
of the coordinates of $g$. 

Without loss of generality, assume 
\[
0<r_{1}\leq r_{2}\leq\cdots<1,\quad r_{j}\to1.
\]
Each $r_{j}$ occurs only finitely many times. We will return to the
original indexing when we construct the operator.

Fix the first $K$ radii and consider the corresponding diagonal operator
on $\mathbb{C}^{K}$. Let 
\[
g=\left(g_{1},\ldots,g_{K}\right)
\]
have positive real coordinates, with $\left\Vert g\right\Vert =1$.

Then 
\[
T^{n}g=\sum^{K}_{j=1}g_{j}r^{n}_{j}e^{in\theta_{j}}e_{j},\quad\left\Vert T^{n}g\right\Vert ^{2}=\sum^{K}_{j=1}g^{2}_{j}r^{2n}_{j},
\]
and the normalized orbit vectors are 
\[
x_{n}:=\frac{T^{n}g}{\left\Vert T^{n}g\right\Vert }=\sum^{K}_{j=1}\frac{g_{j}r^{n}_{j}e^{in\theta_{j}}}{\left(\sum^{K}_{i=1}g^{2}_{i}r^{2n}_{i}\right)^{1/2}}e_{j}.
\]
In particular, 
\begin{equation}
\left|\left\langle x_{n},e_{j}\right\rangle \right|^{2}=\frac{g^{2}_{j}r^{2n}_{j}}{\sum^{K}_{i=1}g^{2}_{i}r^{2n}_{i}}.\label{eq:b-1}
\end{equation}

Fix a positive integer $Q$ and consider the first $QK$ vectors 
\[
x_{0},\ldots,x_{QK-1}
\]
of the normalized orbit, i.e., 
\begin{equation}
\left.\begin{split}x_{0} & =\frac{1}{\left\Vert g\right\Vert }\left(g_{1},g_{2},\dots,g_{K}\right)\\
x_{1} & =\frac{1}{\left\Vert Tg\right\Vert }\left(g_{1}r_{1}e^{i\theta_{1}},g_{2}r_{2}e^{i\theta_{2}},\ldots,g_{K}r_{K}e^{i\theta_{K}}\right)\\
x_{2} & =\frac{1}{\left\Vert T^{2}g\right\Vert }\left(g_{1}r^{2}_{1}e^{i2\theta_{1}},g_{2}r^{2}_{2}e^{i2\theta_{2}},\ldots,g_{K}r^{2}_{K}e^{i2\theta_{K}}\right)\\
 & \;\vdots\\
x_{QK-1} & =\cdots
\end{split}
\right\} \label{eq:b-1a}
\end{equation}
It follows that 
\begin{equation}
\sum^{QK-1}_{n=0}\sum^{K}_{j=1}\left|\left\langle x_{n},e_{j}\right\rangle \right|^{2}=\sum^{QK-1}_{n=0}\left\Vert x_{n}\right\Vert ^{2}=QK.\label{eq:b-2}
\end{equation}

We would like these vectors $x_{0},\ldots,x_{QK-1}$ to form a frame
for $\mathbb{C}^{K}$ with bounds close to each other. To see what
is needed, suppose first that they form a tight frame with bound $A$.
Then 
\begin{equation}
\sum^{QK-1}_{n=0}\left|\left\langle x_{n},h\right\rangle \right|^{2}=A\left\Vert h\right\Vert ^{2},\quad h\in\mathbb{C}^{K}.\label{eq:b-3}
\end{equation}
Taking $h=e_{j}$, we have 
\[
\sum^{QK-1}_{n=0}\left|\left\langle x_{n},e_{j}\right\rangle \right|^{2}=A,\quad1\leq j\leq K.
\]
Summing over $j$ and using \prettyref{eq:b-2} gives $QK=KA$, hence
\[
A=Q.
\]
Back to \prettyref{eq:b-3} with $h=e_{j}$ and \prettyref{eq:b-1},
we obtain 
\begin{equation}
\sum^{QK-1}_{n=0}\frac{g^{2}_{j}r^{2n}_{j}}{\sum^{K}_{i=1}g^{2}_{i}r^{2n}_{i}}=Q,\quad1\leq j\leq K.\label{eq:b-4}
\end{equation}
That is, every column from the rhs of \prettyref{eq:b-1a} has norm
squared $Q$.
\begin{rem*}
For the finite family $x_{0},\ldots,x_{QK-1}$, consider its frame
operator 
\[
Sh=\sum^{QK-1}_{n=0}\left\langle x_{n},h\right\rangle x_{n},\quad h\in\mathbb{C}^{K}.
\]
Relative to the ONB $\left\{ e_{1},\ldots,e_{K}\right\} $, we get
\[
\begin{aligned}\left\langle e_{j},Se_{j}\right\rangle  & =\sum^{QK-1}_{n=0}\left|\left\langle x_{n},e_{j}\right\rangle \right|^{2}=\sum^{QK-1}_{n=0}\frac{g^{2}_{j}r^{2n}_{j}}{\sum^{K}_{k=1}g^{2}_{k}r^{2n}_{k}}.\end{aligned}
\]
So \prettyref{eq:b-4} implies that every diagonal entry of $S$ equals
$Q$. These depend on $r_{j},g_{j}$, but not $\theta_{j}$. The $\theta_{j}$'s
appear in the off-diagonals, 
\[
\left\langle e_{i},Se_{j}\right\rangle =\sum^{QK-1}_{n=0}\frac{g_{i}g_{j}\left(r_{i}r_{j}\right)^{n}}{\sum^{K}_{k=1}g^{2}_{k}r^{2n}_{k}}e^{in\left(\theta_{i}-\theta_{j}\right)},\quad i\ne j.
\]
A tight frame with bound $Q$ would have $S=QI$.

We separate these two requirements. First, we choose positive $g_{j}$
to make the diagonal entries equal to $Q$. Then we choose $\theta_{j}$
to make the frame operator close to $QI$. For the nearly tight frame
we seek, it is enough that the off-diagonal part be small in operator
norm.
\end{rem*}
The following lemma is a special case of matrix scaling with prescribed
row and column sums \cite[Theorem 2]{MR986905}. We include a direct
proof. 
\begin{lem}
\label{lem:2-1} Let $0<r_{1}\leq\cdots\leq r_{K}<1$, and let $Q$
be a positive integer. There are positive numbers $g_{1},\ldots,g_{K}$,
with $\sum^{K}_{j=1}g^{2}_{j}=1$, such that \prettyref{eq:b-4} holds.
\end{lem}

\begin{proof}
In \prettyref{eq:b-4}, replace $g^{2}_{j}$ by $e^{y_{j}}$, then
\[
\frac{\partial}{\partial y_{j}}\log\left(\sum^{K}_{i=1}e^{y_{i}}r^{2n}_{i}\right)=\frac{e^{y_{j}}r^{2n}_{j}}{\sum^{K}_{i=1}e^{y_{i}}r^{2n}_{i}}.
\]
Define
\[
\Phi\left(y_{1},\ldots,y_{K}\right)=\sum^{QK-1}_{n=0}\log\left(\sum^{K}_{i=1}e^{y_{i}}r^{2n}_{i}\right).
\]
Adding the same constant to every $y_{j}$ does not change the fractions
(common factor cancels). So we assume $\sum^{K}_{j=1}y_{j}=0$, and
minimize $\Phi$ on the hyperplane
\[
H=\left\{ y\in\mathbb{R}^{K}:\sum y_{j}=0\right\} .
\]

Since $r_{j}\nearrow1$, we have 
\begin{gather}
\sum^{K}_{j=1}e^{y_{j}}r^{2n}_{j}\geq r^{2n}_{1}e^{\max_{j}y_{j}}\nonumber \\
\Downarrow\nonumber \\
\Phi\left(y\right)\geq QK\max_{j}y_{j}+QK\left(QK-1\right)\log r_{1}.\label{eq:b-5}
\end{gather}

Let $M=\max_{j}y_{j}$. Since $\sum_{j}y_{j}=0$, we have $M\geq0$.
Then 
\begin{equation}
-\left(K-1\right)M\leq-\sum_{i\neq j}y_{i}=y_{j}\leq M,\quad1\leq j\leq K.\label{eq:b-6}
\end{equation}
Combining \prettyref{eq:b-5}--\prettyref{eq:b-6}, 
\[
\left\Vert y\right\Vert \rightarrow\infty\Longrightarrow M\rightarrow\infty\Longrightarrow\Phi\left(y\right)\rightarrow\infty.
\]
Thus, by continuity, $\Phi$ attains a minimum on $H$.

At a minimizer, the partial derivatives of $\Phi$ are equal, but
\[
\begin{aligned}\sum^{K}_{j=1}\frac{\partial\Phi}{\partial y_{j}} & =\sum^{QK-1}_{n=0}\sum^{K}_{j=1}\frac{e^{y_{j}}r^{2n}_{j}}{\sum^{K}_{i=1}e^{y_{i}}r^{2n}_{i}}=\sum^{QK-1}_{n=0}1=QK.\end{aligned}
\]
So we get
\[
\frac{\partial\Phi}{\partial y_{j}}=Q,\quad\forall j.
\]
Finally, setting 
\[
g_{j}=\frac{e^{y_{j}/2}}{\left(\sum^{K}_{i=1}e^{y_{i}}\right)^{1/2}}
\]
we get \prettyref{eq:b-4}.
\end{proof}

\section{Passing to the limit}\label{sec:3}

We now look for a single $g\in\ell^{2}$ so that the infinite version
of \prettyref{eq:b-1a} has the form 
\begin{equation}
\left.\begin{split}x_{0} & =\frac{1}{\sqrt{m_{0}}}\left(g_{1},g_{2},\dots\right)\\
x_{1} & =\frac{1}{\sqrt{m_{1}}}\left(g_{1}r_{1}e^{i\theta_{1}},g_{2}r_{2}e^{i\theta_{2}},\ldots\right)\\
x_{2} & =\frac{1}{\sqrt{m_{2}}}\left(g_{1}r^{2}_{1}e^{i2\theta_{1}},g_{2}r^{2}_{2}e^{i2\theta_{2}},\ldots\right)\\
 & \;\vdots
\end{split}
\right\} \label{eq:3-1a}
\end{equation}
where 
\[
m_{n}:=\sum_{j\geq1}g^{2}_{j}r^{2n}_{j}.
\]

We want to choose $g$ so that every column has squared norm $Q$,
as in the finite case. The next proposition gives such a choice under
the divergence condition.
\begin{prop}
\label{prop:3-1} Suppose $0<r_{1}\leq r_{2}\leq\cdots<1$, $r_{j}\to1$,
and 
\[
\sum_{j\geq1}\left(1-r^{2}_{j}\right)=\infty.
\]
For each $Q\in\mathbb{Z}_{+}$, there exist $g_{j}>0$ with $\sum_{j}g^{2}_{j}=1$
such that,
\begin{equation}
\sum_{n\geq0}\frac{g^{2}_{j}r^{2n}_{j}}{m_{n}}=Q,\quad j\geq1.\label{eq:3-1}
\end{equation}
\end{prop}

\begin{proof}
For each $K$, apply \prettyref{lem:2-1} to the first $K$ radii.
Write $g^{(K)}_{j}$ for the positive coordinates it gives and let
$N_{K}:=QK$. These satisfy 
\[
\sum^{K}_{j=1}(g^{(K)}_{j})^{2}=1
\]
and 
\begin{equation}
m^{(K)}_{n}=\sum^{K}_{j=1}(g^{(K)}_{j})^{2}r^{2n}_{j}.\label{eq:c-2}
\end{equation}
Then \prettyref{eq:b-4} gives 
\[
\sum^{N_{K}-1}_{n=0}\frac{\left(g^{(K)}_{j}\right)^{2}r^{2n}_{j}}{m^{(K)}_{n}}=Q,\quad1\leq j\leq K.
\]

Extend each finite vector $g^{(K)}=(g^{(K)}_{1},\ldots,g^{(K)}_{K})$
to $\ell^{2}$ by setting $g^{(K)}_{j}=0$ for $j>K$, so
\begin{eqnarray*}
g^{\left(1\right)} & = & \left(g^{\left(1\right)}_{1},0,0,\dots\right),\\
g^{\left(2\right)} & = & \left(g^{\left(2\right)}_{1},g^{\left(2\right)}_{2},0,0,\dots\right),\\
g^{\left(3\right)} & = & \left(g^{\left(3\right)}_{1},g^{\left(3\right)}_{2},g^{\left(3\right)}_{3},0,0,\dots\right),\\
 & \vdots
\end{eqnarray*}
Since each vector has norm one, we have $0\leq g^{(K)}_{j}\leq1$
for every $j,K$. 

The first column has a convergent subsequence, so among the corresponding
rows, choose a further subsequence along which the second column converges,
and so on. A diagonal selection gives an increasing sequence $K_{k}\to\infty$
and $g_{j}\geq0$ such that 
\begin{equation}
g^{(K_{k})}_{j}\longrightarrow g_{j}\quad\text{for every fixed }j.\label{eq:3-2}
\end{equation}
The selected rows have the form 
\begin{eqnarray*}
g^{\left(K_{1}\right)} & = & \left(g^{\left(K_{1}\right)}_{1},\ldots,g^{\left(K_{1}\right)}_{K_{1}},0,\ldots\right),\quad\sum(g^{\left(K_{1}\right)}_{j})^{2}=1\\
g^{\left(K_{2}\right)} & = & \left(g^{\left(K_{2}\right)}_{1},\ldots,g^{\left(K_{2}\right)}_{K_{2}},0,\ldots\right),\quad\sum(g^{\left(K_{2}\right)}_{j})^{2}=1\\
g^{\left(K_{3}\right)} & = & \left(g^{\left(K_{3}\right)}_{1},\ldots,g^{\left(K_{3}\right)}_{K_{3}},0,\ldots\right),\quad\sum(g^{\left(K_{3}\right)}_{j})^{2}=1\\
 & \downarrow & \text{coordinatewise}\\
g & = & \left(g_{1},g_{2},g_{3},\dots\right).
\end{eqnarray*}
For simplicity, we continue to use $K$ to index this subsequence. 

For every $J$, 
\[
\sum^{J}_{j=1}g^{2}_{j}=\sum^{J}_{j=1}\lim_{K\to\infty}(g^{(K)}_{j})^{2}=\lim_{K\to\infty}\sum^{J}_{j=1}(g^{(K)}_{j})^{2}\leq1.
\]
Allowing $J\to\infty$, 
\[
\sum_{j}g^{2}_{j}\leq1,
\]
but equality is not yet known.

Using \prettyref{eq:c-2}, 
\begin{align}
1-m^{(K)}_{n} & =\sum^{K}_{j=1}(g^{(K)}_{j})^{2}\left(1-r^{2n}_{j}\right)\nonumber \\
 & =\sum^{J}_{j=1}(g^{(K)}_{j})^{2}\left(1-r^{2n}_{j}\right)+\sum^{K}_{j=J+1}(g^{(K)}_{j})^{2}\left(1-r^{2n}_{j}\right).\label{eq:c-3}
\end{align}

For fixed $n,J$, the first sum in \prettyref{eq:c-3} converges as
$K\to\infty$, since every fixed coordinate converges; the second
satisfies 
\[
0\leq\sum^{K}_{j=J+1}(g^{(K)}_{j})^{2}\left(1-r^{2n}_{j}\right)\leq1-r^{2n}_{J+1}\xrightarrow{\;J\rightarrow\infty\;}0,
\]
independently of $K$. Thus we take limits, first in $K$, then in
$J$, and get 
\[
1-m^{\left(K\right)}_{n}\longrightarrow\sum_{j\geq1}g^{2}_{j}\left(1-r^{2n}_{j}\right).
\]
That is, 
\[
\begin{aligned}m^{(K)}_{n}\longrightarrow & 1-\sum\nolimits_{j\geq1}g^{2}_{j}(1-r^{2n}_{j})\\
 & \quad=\underbrace{\left(1-\sum\nolimits_{j\geq1}g^{2}_{j}\right)}_{a}+\sum\nolimits_{j\geq1}g^{2}_{j}r^{2n}_{j}=a+m_{n}.
\end{aligned}
\]

We claim that $a=0$. Note that $a+m_{0}=1$ and $a+m_{n}>0$ for
every $n$.

For each fixed $j,n$, the convergence above gives 
\[
\frac{(g^{(K)}_{j})^{2}r^{2n}_{j}}{m^{(K)}_{n}}\longrightarrow\frac{g^{2}_{j}r^{2n}_{j}}{a+m_{n}},\quad K\rightarrow\infty.
\]
For every fixed $M$, once $N_{K}>M$, 
\[
\sum^{M}_{n=0}\frac{(g^{(K)}_{j})^{2}r^{2n}_{j}}{m^{(K)}_{n}}\leq\sum^{N_{K}-1}_{n=0}\frac{(g^{(K)}_{j})^{2}r^{2n}_{j}}{m^{(K)}_{n}}=Q.
\]
Letting $K\to\infty$, then $M\rightarrow\infty$ gives 
\begin{equation}
\sum_{n\geq0}\frac{g^{2}_{j}r^{2n}_{j}}{a+m_{n}}\leq Q.\label{eq:c-5}
\end{equation}

We first show that 
\[
\sum^{J}_{j=1}g^{2}_{j}<1
\]
for every $J$. Otherwise, for some $J$, $\sum^{J}_{j=1}g^{2}_{j}=1$.
Then 
\[
a=1-\sum g^{2}_{j}=0,
\]
and all $g_{j}$ with $j>J$ vanish. Hence 
\[
m_{n}=\sum^{J}_{j=1}g^{2}_{j}r^{2n}_{j}>0.
\]
But then 
\[
\infty=\sum_{n\geq0}\frac{\sum^{J}_{j=1}g^{2}_{j}r^{2n}_{j}}{m_{n}}=\sum^{J}_{j=1}\left(\sum_{n\geq0}\frac{g^{2}_{j}r^{2n}_{j}}{m_{n}}\right)\leq JQ,
\]
a contradiction, where the inequality follows from \prettyref{eq:c-5}
with $a=0$. 

Now fix $j$ and choose $J\geq j$ so that $r_{j}<r_{J+1}$. This
is possible since $r_{i}\to1$ and $r_{j}<1$. Let 
\[
b=1-\sum^{J}_{i=1}g^{2}_{i}>0.
\]
Returning to our selected rows, 
\[
\sum^{K}_{i=J+1}(g^{(K)}_{i})^{2}=1-\sum^{J}_{i=1}(g^{(K)}_{i})^{2}\longrightarrow b.
\]
Thus, for sufficiently large $K$, 
\[
\sum^{K}_{i=J+1}(g^{(K)}_{i})^{2}>\frac{b}{2}.
\]
Since the radii are increasing, for every $n\geq0$, 
\[
\begin{aligned}m^{(K)}_{n} & =\sum^{K}_{i=1}(g^{(K)}_{i})^{2}r^{2n}_{i}\geq\sum^{K}_{i=J+1}(g^{(K)}_{i})^{2}r^{2n}_{i}\geq\frac{b}{2}r^{2n}_{J+1}.\end{aligned}
\]
This implies 
\[
0\leq\frac{(g^{(K)}_{j})^{2}r^{2n}_{j}}{m^{(K)}_{n}}\leq\frac{2}{b}\left(\frac{r_{j}}{r_{J+1}}\right)^{2n}.
\]
Dominated convergence then gives 
\[
\lim_{K\rightarrow\infty}\sum^{N_{K}-1}_{n=0}\frac{(g^{(K)}_{j})^{2}r^{2n}_{j}}{m^{(K)}_{n}}=\sum_{n\geq0}\frac{g^{2}_{j}r^{2n}_{j}}{a+m_{n}}.
\]
The sum on the left equals $Q$ for every sufficiently large $K$,
so 
\[
\sum_{n\geq0}\frac{g^{2}_{j}r^{2n}_{j}}{a+m_{n}}=Q,\quad j\geq1.
\]
In particular, $g_{j}>0$ for every $j$.

It follows that 
\begin{align}
Q\sum_{j\geq1}\left(1-r^{2}_{j}\right) & =\sum_{j}\left(1-r^{2}_{j}\right)\sum_{n\geq0}\frac{g^{2}_{j}r^{2n}_{j}}{a+m_{n}}\nonumber \\
 & =\sum_{n\geq0}\frac{\sum_{j}g^{2}_{j}\left(1-r^{2}_{j}\right)r^{2n}_{j}}{a+m_{n}}\nonumber \\
 & =\sum_{n\geq0}\frac{m_{n}-m_{n+1}}{a+m_{n}}.\label{eq:3-7}
\end{align}

By hypothesis, $\sum_{j\geq1}\left(1-r^{2}_{j}\right)=\infty$. If
$a>0$, then \prettyref{eq:3-7} becomes 
\[
\infty=\begin{aligned}\sum_{n\geq0}\frac{m_{n}-m_{n+1}}{a+m_{n}} & \leq\frac{1}{a}\sum_{n\geq0}\left(m_{n}-m_{n+1}\right)\leq\frac{m_{0}}{a}\leq\frac{1}{a}\end{aligned}
,
\]
a contradiction. Here, we used 
\[
m_{n}\leq m_{0}=\sum g^{2}_{j}=1-a<1.
\]

Thus $a=0$ and $\sum_{j}g^{2}_{j}=1$. The assertion \prettyref{eq:3-1}
holds. 
\end{proof}

\section{Choosing the angles}\label{sec:4}

We have chosen the positive coordinates of $g$ in \prettyref{sec:3}.
The angles $\theta_{j}$ of the eigenvalues are still free. We first
choose them for the finite system in \prettyref{eq:b-1a}, then pass
to the limit as in \prettyref{eq:3-1a}.

Fix $K,Q\in\mathbb{Z}_{+}$. Take the coordinates $g^{(K)}_{j}$ given
by \prettyref{lem:2-1}. Write 
\[
a_{j}\left(n\right)=\frac{g^{(K)}_{j}r^{n}_{j}}{\sqrt{m^{(K)}_{n}}},\quad1\leq j\leq K,\quad0\leq n<QK.
\]
The normalized orbit vectors have the form 
\[
x_{n}=\left(a_{1}\left(n\right)e^{in\theta_{1}},\ldots,a_{K}\left(n\right)e^{in\theta_{K}}\right),\quad0\leq n<QK.
\]
We know that 
\begin{align}
\sum^{K}_{j=1}a_{j}\left(n\right)^{2} & =1,\label{eq:4-1a}\\
\sum^{QK-1}_{n=0}a_{j}\left(n\right)^{2} & =Q.\label{eq:4-1}
\end{align}
The first identity \prettyref{eq:4-1a} says that every $x_{n}$ has
norm one, \prettyref{eq:4-1} comes from our choice of $g^{(K)}_{j}$.

Recall the frame operator from \prettyref{sec:2}, 
\[
Sh=\sum^{QK-1}_{n=0}\left\langle x_{n},h\right\rangle x_{n},\quad h\in\mathbb{C}^{K}.
\]
By \prettyref{eq:4-1}, 
\[
\left\langle e_{j},Se_{j}\right\rangle =\sum^{QK-1}_{n=0}\left|\left\langle x_{n},e_{j}\right\rangle \right|^{2}=Q,
\]
i.e., under the ONB $\left(e_{j}\right)$, all diagonal entries of
$S$ are $Q$.

The angles only occur in its off-diagonal entries. We want to choose
them so that 
\[
\left\Vert S-QI\right\Vert \leq\varepsilon Q.
\]
Then 
\[
\left|\left\langle h,\left(S-QI\right)h\right\rangle \right|\leq\left\Vert h\right\Vert \left\Vert \left(S-QI\right)h\right\Vert \leq\varepsilon Q\left\Vert h\right\Vert ^{2},
\]
and so 
\begin{equation}
Q\left(1-\varepsilon\right)\left\Vert h\right\Vert ^{2}\leq\sum^{QK-1}_{n=0}\left|\left\langle x_{n},h\right\rangle \right|^{2}\leq Q\left(1+\varepsilon\right)\left\Vert h\right\Vert ^{2}.\label{eq:4-2}
\end{equation}

We will not estimate the off-diagonal entries individually. The following
result chooses all the angles at once, using only the two identities
in \prettyref{eq:4-1a}--\prettyref{eq:4-1}. The special form of
$a_{j}\left(n\right)$ is not needed.
\begin{prop}
\label{prop:4-1} For every $0<\varepsilon<1$, there is a positive
integer $Q$, depending only on $\varepsilon$, such that for every
$K$ and every nonnegative real array satisfying \prettyref{eq:4-1a}--\prettyref{eq:4-1},
the angles $\theta_{1},\ldots,\theta_{K}$ can be chosen so that \prettyref{eq:4-2}
holds. 
\end{prop}

\begin{proof}
We use \cite[Theorem 3.8]{MR4013969}. Let $\left\{ u_{i}\right\} _{i\in I}$
be a family of unit vectors in a Hilbert space $H$, satisfying 
\[
\sum_{i\in I}\left|\left\langle u_{i},z\right\rangle \right|^{2}\leq B\left\Vert z\right\Vert ^{2},\quad z\in H.
\]
Suppose $J_{1},\ldots,J_{K}$ are disjoint subsets of $I$, each containing
at least $CB\varepsilon^{-4}$ indices, where $C$ is an absolute
constant. The theorem allows us to choose $i_{j}\in J_{j}$, one from
each group, so that 
\[
\left(1-\varepsilon\right)\sum^{K}_{j=1}\left|c_{j}\right|^{2}\leq\left\Vert \sum^{K}_{j=1}c_{j}u_{i_{j}}\right\Vert ^{2}\leq\left(1+\varepsilon\right)\sum^{K}_{j=1}\left|c_{j}\right|^{2}
\]
for all scalars $c_{1},\ldots,c_{K}$.

In our problem, we will take $H=\mathbb{C}^{QK}$ and construct $K$
groups of unit vectors, with $QK$ vectors in each group. We want
the complete family to have Bessel bound $B=K$. If this can be done,
the group size condition becomes 
\[
QK\geq CK\varepsilon^{-4},
\]
which holds once $Q\geq C\varepsilon^{-4}$, independently of $K$.

We now construct those vectors. For $h=\sum^{K}_{j=1}c_{j}e_{j}$,
\begin{align*}
\left\langle h,Sh\right\rangle  & =\sum^{QK-1}_{n=0}\left\langle h,x_{n}\right\rangle \left\langle x_{n},h\right\rangle \\
 & =\sum^{QK-1}_{n=0}\left|\left\langle x_{n},h\right\rangle \right|^{2}=\sum^{QK-1}_{n=0}\left|\sum^{K}_{j=1}c_{j}a_{j}\left(n\right)e^{-in\theta_{j}}\right|^{2}.
\end{align*}
This suggests the following: 

Fix $j$. As $n$ varies, the coefficients of $c_{j}$ form the vector
\[
\frac{1}{\sqrt{Q}}\left(a_{j}\left(n\right)e^{-in\theta_{j}}\right)^{QK-1}_{n=0}\in\mathbb{C}^{QK}.
\]
By \prettyref{eq:4-1}, this vector has norm one for every choice
of $\theta_{j}$. Thus each possible angle gives one candidate vector
for the $j$-th group.

We allow $\theta_{j}$ to take the $QK$ equally spaced values $2\pi l/(QK)$,
where $0\leq l<QK$. This gives 
\[
u_{j,l}=\frac{1}{\sqrt{Q}}\left(a_{j}\left(n\right)e^{-2\pi inl/\left(QK\right)}\right)^{QK-1}_{n=0}
\]
with $1\leq j\leq K$, $0\leq l<QK$. So, for each fixed $j$, these
are the $QK$ candidates in the $j$-th group. 

Note that, by \prettyref{eq:4-1}, 
\[
\left\Vert u_{j,l}\right\Vert ^{2}=\frac{1}{Q}\sum^{QK-1}_{n=0}a_{j}\left(n\right)^{2}=1.
\]

It remains to verify the Bessel bound for all these vectors together.
Let $z\in\mathbb{C}^{QK}$. Orthogonality of the $QK$-th roots of
unity gives 
\[
\begin{aligned}\sum^{K}_{j=1}\sum^{QK-1}_{l=0}\left|\left\langle u_{j,l},z\right\rangle \right|^{2} & =\frac{1}{Q}\sum^{K}_{j=1}\sum^{QK-1}_{l=0}\left|\sum^{QK-1}_{n=0}a_{j}\left(n\right)z\left(n\right)e^{2\pi inl/(QK)}\right|^{2}\\
 & =K\sum^{K}_{j=1}\sum^{QK-1}_{n=0}a_{j}\left(n\right)^{2}\left|z\left(n\right)\right|^{2}\\
 & =K\sum^{QK-1}_{n=0}\left|z\left(n\right)\right|^{2}=K\left\Vert z\right\Vert ^{2}.
\end{aligned}
\]
The last equality uses \prettyref{eq:4-1a}. Thus the whole family
has Bessel bound $K$. 

Choose a positive integer $Q\geq C\varepsilon^{-4}$. All the hypotheses
of the selector theorem are now satisfied. It gives indices $l_{1},\ldots,l_{K}$,
with $0\leq l_{j}<QK$, such that 
\[
\left(1-\varepsilon\right)\sum^{K}_{j=1}\left|c_{j}\right|^{2}\leq\left\Vert \sum^{K}_{j=1}c_{j}u_{j,l_{j}}\right\Vert ^{2}\leq\left(1+\varepsilon\right)\sum^{K}_{j=1}\left|c_{j}\right|^{2}
\]
for all scalars $c_{j}$.

Set $\theta_{j}=2\pi l_{j}/(QK)$. Returning to the formula for $\left\langle h,Sh\right\rangle $,
we obtain 
\[
\begin{aligned}\left\langle h,Sh\right\rangle  & =\sum^{QK-1}_{n=0}\left|\sum^{K}_{j=1}c_{j}a_{j}\left(n\right)e^{-in\theta_{j}}\right|^{2}\\
 & =Q\sum^{QK-1}_{n=0}\left|\sum^{K}_{j=1}c_{j}\frac{a_{j}\left(n\right)e^{-2\pi inl_{j}/\left(QK\right)}}{\sqrt{Q}}\right|^{2}\\
 & =Q\sum^{QK-1}_{n=0}\left|\sum^{K}_{j=1}c_{j}u_{j,l_{j}}\left(n\right)\right|^{2}\\
 & =Q\left\Vert \sum^{K}_{j=1}c_{j}u_{j,l_{j}}\right\Vert ^{2}.
\end{aligned}
\]
Consequently, 
\[
\begin{aligned}\left|\left\langle h,\left(S-QI\right)h\right\rangle \right| & =Q\left|\left\Vert \sum^{K}_{j=1}c_{j}u_{j,l_{j}}\right\Vert ^{2}-\sum^{K}_{j=1}\left|c_{j}\right|^{2}\right|\\
 & \leq\varepsilon Q\sum^{K}_{j=1}\left|c_{j}\right|^{2}=\varepsilon Q\left\Vert h\right\Vert ^{2}.
\end{aligned}
\]
Since $S-QI$ is self-adjoint, 
\[
\left\Vert S-QI\right\Vert =\sup_{\left\Vert h\right\Vert =1}\left|\left\langle h,\left(S-QI\right)h\right\rangle \right|\leq\varepsilon Q.
\]
This proves \prettyref{eq:4-2}. 
\end{proof}

\begin{rem}
A more recent selector theorem \cite[Corollary 4.8]{BownikSelector2026}
improves the group size requirement from $CB\varepsilon^{-4}$ to
$CB\varepsilon^{-2}$. So $Q$ may be chosen of order $\varepsilon^{-2}$. 
\end{rem}

By \prettyref{lem:2-1} and \prettyref{prop:4-1}, for each $K$ we
can choose the coordinates and angles so that the first $QK$ normalized
orbit vectors form a frame for $\mathbb{C}^{K}$ with bounds $Q\left(1-\varepsilon\right)$
and $Q\left(1+\varepsilon\right)$. Here $Q$ is independent of $K$
and of the prescribed radii. We next pass to the infinite orbit. 
\begin{thm}
\label{thm:4-3} Let $0<r_{j}<1$ and $r_{j}\to1$. The following
are equivalent. 
\begin{enumerate}
\item \label{enu:4-3a}There exist $\lambda_{j}\in\mathbb{C}$ with $\left|\lambda_{j}\right|=r_{j}$
and $g\in\ell^{2}$ such that the normal operator $T$ defined by
$Te_{j}=\lambda_{j}e_{j}$ has a normalized orbit 
\[
\left\{ \frac{T^{n}g}{\left\Vert T^{n}g\right\Vert }\right\} _{n\geq0}
\]
which is a frame for $\ell^{2}$.
\item \label{enu:4-3b}The series $\sum_{j\geq1}\left(1-r^{2}_{j}\right)$
diverges. 
\end{enumerate}
Moreover, assume 
\begin{equation}
\sum_{j\geq1}\left(1-r^{2}_{j}\right)=\infty.\label{eq:4-4a}
\end{equation}
Then, for every $\varepsilon\in\left(0,1\right)$, there exists $Q=Q\left(\varepsilon\right)\in\mathbb{Z}_{+}$,
for which $T$ and a unit vector $g$ can be chosen so that 
\begin{equation}
Q\left(1-\varepsilon\right)\left\Vert h\right\Vert ^{2}\leq\sum_{n\geq0}\left|\left\langle \frac{T^{n}g}{\left\Vert T^{n}g\right\Vert },h\right\rangle \right|^{2}\leq Q\left(1+\varepsilon\right)\left\Vert h\right\Vert ^{2},\quad h\in\ell^{2}.\label{eq:4-4}
\end{equation}
\end{thm}

\begin{proof}
\eqref{enu:4-3b}$\Longrightarrow$\eqref{enu:4-3a}. Assume \prettyref{eq:4-4a}
holds. Fix $\varepsilon\in\left(0,1\right)$, and choose $Q$ as in
\prettyref{prop:4-1}.

For each $K$, take the finite coordinates $g^{\left(K\right)}_{j}$
used in the proof of \prettyref{prop:3-1}. Retain the subsequence
for which 
\[
g^{\left(K\right)}_{j}\longrightarrow g_{j},\quad m^{\left(K\right)}_{n}\longrightarrow m_{n}
\]
for every fixed $j,n$. The same proof gives positive limiting coordinates
with $\sum_{j}g^{2}_{j}=1$ and 
\[
\sum_{n\geq0}\frac{g^{2}_{j}r^{2n}_{j}}{m_{n}}=Q,\quad j\geq1.
\]

For each finite problem, set 
\[
a^{(K)}_{j}\left(n\right)=\frac{g^{\left(K\right)}_{j}r^{n}_{j}}{\sqrt{m^{\left(K\right)}_{n}}},\quad1\leq j\leq K,\quad0\leq n<QK.
\]
As in \prettyref{eq:4-1a}--\prettyref{eq:4-1}, these numbers satisfy
\[
\sum^{K}_{j=1}a^{(K)}_{j}\left(n\right)^{2}=1,\quad\sum^{QK-1}_{n=0}a^{(K)}_{j}\left(n\right)^{2}=Q.
\]
By \prettyref{prop:4-1}, there exist $\theta^{(K)}_{j}$ for which
the corresponding finite orbit vectors satisfy \prettyref{eq:4-2}.

Let 
\[
u^{\left(K\right)}_{j}\left(n\right)=\frac{g^{\left(K\right)}_{j}r^{n}_{j}}{\sqrt{Qm^{(K)}_{n}}}e^{-in\theta^{\left(K\right)}_{j}},\quad0\leq n<QK.
\]
These are the columns of the finite analysis matrix divided by $\sqrt{Q}$.
The frame inequalities in \prettyref{eq:4-2} therefore give 
\begin{equation}
\left(1-\varepsilon\right)\sum^{K}_{j=1}\left|c_{j}\right|^{2}\leq\left\Vert \sum^{K}_{j=1}c_{j}u^{\left(K\right)}_{j}\right\Vert ^{2}\leq\left(1+\varepsilon\right)\sum^{K}_{j=1}\left|c_{j}\right|^{2}.\label{eq:4-5}
\end{equation}

We regard each $u^{(K)}_{j}$ as a vector in $\ell^{2}\left(\mathbb{N}_{0}\right)$
by setting its coordinates equal to zero for $n\geq QK$. By \prettyref{eq:b-4},
\[
\left\Vert u^{(K)}_{j}\right\Vert ^{2}=\frac{(g^{(K)}_{j})^{2}}{Q}\sum^{QK-1}_{n=0}\frac{r^{2n}_{j}}{m^{\left(K\right)}_{n}}=1.
\]

The coordinates and moments converge along our chosen subsequence.
Passing to a further subsequence, we may assume that $e^{i\theta^{(K)}_{j}}\to e^{i\theta_{j}}$
for every fixed $j$. This follows by a diagonal selection, since
the unit circle is compact.

Define 
\[
u_{j}\left(n\right)=\frac{g_{j}r^{n}_{j}}{\sqrt{Qm_{n}}}e^{-in\theta_{j}},\quad n\geq0.
\]
For every fixed $j,n$, we have $u^{\left(K\right)}_{j}\left(n\right)\to u_{j}\left(n\right)$.
Also, \prettyref{prop:3-1} gives 
\begin{equation}
\left\Vert u_{j}\right\Vert ^{2}=\frac{g^{2}_{j}}{Q}\sum_{n\geq0}\frac{r^{2n}_{j}}{m_{n}}=1.\label{eq:4-6}
\end{equation}

We need to pass \prettyref{eq:4-5} to the limiting columns. Coordinatewise
convergence alone is not enough, since some of the norm could disappear
at late times. But both the finite columns and their limits have norm
one.

For each fixed $j$, coordinatewise convergence and the uniform norm
bound imply that $u^{\left(K\right)}_{j}$ converges weakly to $u_{j}$.
Indeed, convergence holds when tested against finitely supported vectors,
and then against every vector by approximation. Weak convergence and
equality of the norms give 
\[
\begin{aligned}\left\Vert u^{(K)}_{j}-u_{j}\right\Vert ^{2} & =\left\Vert u^{(K)}_{j}\right\Vert ^{2}+\left\Vert u_{j}\right\Vert ^{2}-2\Re\left\langle u^{(K)}_{j},u_{j}\right\rangle \\
 & \longrightarrow1+1-2=0.
\end{aligned}
\]
Thus every fixed column converges in norm.

Take finitely supported scalars $c_{j}$. For sufficiently large $K$,
all their indices lie among the first $K$ coordinates. We can pass
to the limit in \prettyref{eq:4-5}, since only finitely many columns
occur. This gives 
\begin{equation}
\left(1-\varepsilon\right)\sum_{j}\left|c_{j}\right|^{2}\leq\left\Vert \sum_{j}c_{j}u_{j}\right\Vert ^{2}\leq\left(1+\varepsilon\right)\sum_{j}\left|c_{j}\right|^{2}.\label{eq:4-7}
\end{equation}
The upper inequality shows that the map 
\[
\sum_{j}c_{j}e_{j}\longmapsto\sum_{j}c_{j}u_{j},
\]
initially defined on finitely supported sequences, extends to a bounded
operator on $\ell^{2}$. The lower inequality passes to every $c\in\ell^{2}$
by finite truncation.

Set 
\[
\lambda_{j}=r_{j}e^{i\theta_{j}},\quad Te_{j}=\lambda_{j}e_{j},\quad g=\sum_{j\geq1}g_{j}e_{j}.
\]
Then $T$ is normal, and $g$ has norm $1$. For each $n$, 
\[
T^{n}g=\sum_{j\geq1}g_{j}r^{n}_{j}e^{in\theta_{j}}e_{j},\quad\left\Vert T^{n}g\right\Vert ^{2}=\sum_{j\geq1}g^{2}_{j}r^{2n}_{j}=m_{n}>0.
\]

Let $h=\sum_{j}c_{j}e_{j}$ have finite support. Then, 
\[
\begin{aligned}\left\langle \frac{T^{n}g}{\left\Vert T^{n}g\right\Vert },h\right\rangle  & =\sum_{j}c_{j}\frac{g_{j}r^{n}_{j}}{\sqrt{m_{n}}}e^{-in\theta_{j}}=\sqrt{Q}\sum_{j}c_{j}u_{j}\left(n\right).\end{aligned}
\]
Now \prettyref{eq:4-7} gives 
\[
Q\left(1-\varepsilon\right)\left\Vert h\right\Vert ^{2}\leq\sum_{n\geq0}\left|\left\langle \frac{T^{n}g}{\left\Vert T^{n}g\right\Vert },h\right\rangle \right|^{2}\leq Q\left(1+\varepsilon\right)\left\Vert h\right\Vert ^{2}.
\]
The upper estimate allows this inequality to extend from finitely
supported vectors to every $h\in\ell^{2}$. Undoing the preliminary
reordering gives an operator with the radii in their original order.

\eqref{enu:4-3a}$\Longrightarrow$\eqref{enu:4-3b}. Suppose the
normalized orbit 
\[
x_{n}=\frac{T^{n}g}{\left\Vert T^{n}g\right\Vert },\quad n\geq0,
\]
is a frame with upper bound $B$. We will only use the upper frame
inequality.

Suppose instead that 
\[
\sum_{j\geq1}\left(1-r^{2}_{j}\right)<\infty.
\]
We have $\left\Vert T\right\Vert =\sup\left|\lambda_{j}\right|=1$,
so 
\[
D=I-T^{*}T\geq0,
\]
and 
\[
tr\left(D\right)=\sum\left\langle e_{j},De_{j}\right\rangle =\sum_{j\geq1}(1-r^{2}_{j})<\infty.
\]

Moreover, 
\begin{equation}
\left\langle x_{n},Dx_{n}\right\rangle =1-\frac{\left\Vert T^{n+1}g\right\Vert ^{2}}{\left\Vert T^{n}g\right\Vert ^{2}}.\label{eq:4-8}
\end{equation}

We first show that the sum of these quantities is finite. Indeed,
\[
\begin{aligned}\sum_{n\geq0}\left\langle x_{n},Dx_{n}\right\rangle  & =\sum_{n\geq0}\left\Vert D^{1/2}x_{n}\right\Vert ^{2}\\
 & =\sum_{n\geq0}\sum_{j\geq1}\left|\left\langle x_{n},D^{1/2}e_{j}\right\rangle \right|^{2}\\
 & =\sum_{j\geq1}\sum_{n\geq0}\left|\left\langle x_{n},D^{1/2}e_{j}\right\rangle \right|^{2}\\
 & \leq B\sum_{j\geq1}\left\Vert D^{1/2}e_{j}\right\Vert ^{2}\\
 & =B\cdot tr\left(D\right)<\infty.
\end{aligned}
\]
The inequality is the upper frame bound applied to $D^{1/2}e_{j}$.

On the other hand, \prettyref{eq:4-8} shows that this sum cannot
be finite. To see this, write 
\[
a_{n}=\left\langle x_{n},Dx_{n}\right\rangle .
\]
We have $0\leq a_{n}<1$. By the preceding estimate, 
\[
\sum_{n\geq0}a_{n}<\infty.
\]
Choose $N$ so large that 
\[
\sum_{n\geq N}a_{n}<\frac{1}{2}.
\]

For $L>N$, \prettyref{eq:4-8} gives 
\begin{equation}
\begin{aligned}\frac{\left\Vert T^{L}g\right\Vert ^{2}}{\left\Vert T^{N}g\right\Vert ^{2}} & =\prod^{L-1}_{n=N}\left(1-a_{n}\right)\geq1-\sum^{L-1}_{n=N}a_{n}>\frac{1}{2}.\end{aligned}
\label{eq:4-10}
\end{equation}
Note that 
\[
\left\Vert T^{n}g\right\Vert ^{2}=\sum_{j\geq1}\left|g_{j}\right|^{2}r^{2n}_{j}\longrightarrow0.
\]
So $\left\Vert T^{L}g\right\Vert \to0$, and $\left\Vert T^{N}g\right\Vert >0$.
The lhs of \eqref{eq:4-10} therefore tends to zero, a contradiction.
Hence \eqref{eq:4-4a} holds.
\end{proof}

\bibliographystyle{amsalpha}
\bibliography{ref}

\end{document}